\documentclass[11pt]{amsart}

\usepackage[T1]{fontenc}
\usepackage[utf8]{inputenc}
\usepackage{amsmath,amssymb,amsthm}
\usepackage{mathrsfs}

\title[A proof of Alexander's conjecture]
{A proof of Alexander's conjecture on an inequality of Cassels}
\author{Myriam Ouna\"{\i}es}
\address{Institut de Recherche Math\'ematique Avanc\'ee (IRMA)\\
Universit\'e de Strasbourg, France}
\email{myriam.ounaies@math.unistra.fr}
\subjclass[2020]{30C80, 26D15}
\keywords{Cassels inequality, Alexander conjecture, Schur inequality,
complex inequalities, extremal configurations}

\newtheorem{theorem}{Theorem}

\newtheorem{corollary}[theorem]{Corollary}
\theoremstyle{definition}

\theoremstyle{remark}

\begin{document}

\begin{abstract}
Let $z_1,\dots,z_n$ be complex numbers with $|z_j|\le \rho$, where $\rho>1$.
Cassels proved that, under an additional restriction on $\rho$,
the inequality
\[
\prod_{j\ne k}\bigl|1-\overline{z_j}z_k\bigr|
\le
\left(\frac{\rho^{2n}-1}{\rho^2-1}\right)^{\!n}
\]
holds.
In a subsequent note, Alexander conjectured that this inequality is in fact valid without any
restriction on $\rho$.
In this paper, we confirm Alexander's conjecture.
\end{abstract}

\maketitle

\section{Introduction}

We will use the following notations for the unit disc and the unit circle :
\[\mathbb{D}=\{z\in\mathbb{C}:|z|<1\},\ \ \overline{\mathbb{D}}=\{z\in\mathbb{C}:|z|\le 1\},\ \  
\mathbb{T}=\{z\in\mathbb{C}:|z|=1\}.\]

Let $z_1,\dots,z_n$ be in $\mathbb{D}$.
A classical inequality, going back to Schur~\cite{Schur1918}, states that
\begin{equation}\label{eq:schur}
\prod_{j\ne k}\bigl|1-\overline{z_j}z_k\bigr|\le n^n,
\end{equation}
with equality if and only if the $z_j$ are the vertices of a regular $n$-gon
inscribed in $\Bbb T$.

In connection with a problem of Schinzel and Zassenhaus, Cassels~\cite{Cassels1966}
extended~\eqref{eq:schur} with the following:

\begin{theorem}[Cassels]
Let $\rho>1$ and let $z_1,\dots,z_n\in \{z\in \Bbb C :|z|\le \rho\}$.
Suppose
\begin{equation}\label{condition}
\cos(\pi/n)\le \frac{\rho^2}{\rho^4-\rho^2+1}.
\end{equation}
Then
\begin{equation}\label{eq:cassels}
\prod_{j\ne k}\bigl|1-\overline{z_j}z_k\bigr|
\le
\left(\frac{\rho^{2n}-1}{\rho^2-1}\right)^{\!n},
\end{equation}
with equality if and only if the $z_j$ are the vertices of a regular $n$-gon
inscribed in the circle $\{z\in \Bbb C :|z|= \rho\}$.
\end{theorem}

Alexander~\cite{Alexander1972} observed that condition \eqref{condition} may be replaced
by the weaker assumption
\[
\cos(\pi/n)\le \frac{2\rho^2}{\rho^4+1},
\]
and conjectured that \eqref{eq:cassels} should in fact hold for all $\rho>1$,
without assuming \eqref{condition}.

In a recent work, Dubickas~\cite{Dubickas2020} revisited Cassels' inequality by
observing that, for $z_j=\rho\,\omega_j$ with $\omega_j\in \Bbb T$ ($j=1,\dots,n$), one has
\[
\prod_{j\neq k}\bigl|1-\overline{z_j}z_k\bigr|
=\rho^{n(n-1)}\prod_{j<k}\left((\rho-\rho^{-1})^2+|\omega_j-\omega_k|^2\right).
\]
He conjectured that the elementary symmetric functions of the quantities
$\{|\omega_j-\omega_k|^2,\ j<k\}$ attain their maximum for the vertices of a regular
$n$-gon inscribed in $\mathbb{T}$, and proved this for the elementary symmetric functions of degree
$1$ through $4$. His conjecture implies
Alexander's.

In the present paper, we confirm Alexander's conjecture.
Our approach relies on an auxiliary additive inequality and a monotonicity argument.
% =========================
% Main result
% =========================
\section{Main result}

\begin{theorem}\label{thm:main}

Let $\rho>1$ and let $z_1,\dots,z_n\in \{z\in \Bbb C :|z|\le \rho\}$. Then the following inequality holds:
\begin{equation}\label{ineqmain2}
\prod_{j\neq k}\vert 1-\overline{z_j}z_k\vert\le (\rho^{2n}-1)^n(\rho^2-1)^{-n}.
\end{equation}
Equality is attained if and only if the $z_j$ are the vertices of a regular $n$-gon inscribed in the circle $\{z\in \Bbb C :|z|= \rho\}$.

\end{theorem}

% =========================
% Auxiliary result
% =========================
\section{An auxiliary inequality}

\begin{theorem}\label{thm:additive}

Let $z_1,\dots,z_n\in \mathbb{D}$ and denote $\Lambda=(-1)^n\prod_{j=1}^n z_j$. Then we have the following inequality:
\begin{equation}\label{ineqmain}
\sum_{j,k}\frac1{1-\overline{z_j}z_k}\ge \frac{n^2}{1-\vert \Lambda\vert^2}.
\end{equation}
Equality is attained if and only if the $z_j$ are the $n$ roots of the polynomial $P(z)=z^n+\Lambda$. 

\end{theorem}

\begin{corollary}\label{cor:key}

Let $\omega_1,\dots,\omega_n\in \Bbb T$ and $a\in (0,1)$ . Then the following inequality holds:
\begin{equation}\label{ineqmain1}
\sum_{j,k}\frac1{1-a\overline{\omega_j} \omega_k}\ge \frac{n^2}{1-a^{n}}.
\end{equation}
Equality is attained if and only if the $\omega_j$ are the vertices of a regular $n$-gon inscribed in $\Bbb T$.

\end{corollary}

\begin{proof}
Apply Theorem~\ref{thm:additive} to $z_j=\sqrt{a}\,\omega_j\in \Bbb D$ $(j=1,\dots,n)$. Then \[
1-\overline{z_j}z_k=1-a\,\overline{\omega_j}\omega_k,\qquad 
|Z|^2=\Bigl|\prod_{j=1}^n z_j\Bigr|^2=a^n.
\]
This gives \eqref{ineqmain1}. The equality case follows from that of
Theorem~\ref{thm:additive}.
\end{proof}

% =========================
% Proof of main theorem
% =========================
\section{Proof Theorem~\ref{thm:main}}

Fix an index $k$. The map
\[
z_k\longmapsto \prod_{j\neq k}(1-\overline{z_j}z_k)
\]
is holomorphic in the disc $\{|z_k|\le\rho\}$. By the maximum modulus principle,
the quantity $\prod_{j\neq k}|1-\overline{z_j}z_k|$ is maximized when $|z_k|=\rho$.
Iterating this argument for $k=1,\dots,n$ (applied to the factors involving $z_k$
in the product) shows that
\begin{equation}\label{ineqmain4}
\prod_{j\neq k}|1-\overline{z_j}z_k|
\le
\max_{|\omega_1|=\cdots=|\omega_n|=1}
\prod_{j\neq k}|1-\rho^2\overline{\omega_j}\omega_k|,
\end{equation}
and equality can occur only if $|z_1|=\cdots=|z_n|=\rho$.

Setting $z_j=\rho\,\omega_j$ with $\omega_j\in \Bbb T$ ($j=1,\dots,n)$, \eqref{ineqmain2} is equivalent to
\begin{equation}\label{ineqmain3}
\prod_{j,k}\bigl|\rho^{-2}-\overline{\omega_j}\omega_k\bigr|
\le
(1-\rho^{-2n})^n .
\end{equation}
We fix $\omega_1,\dots,\omega_n$ on $\Bbb T$ and we define, for $a\in [0,1)$,
\[g_\omega(a)=n\log(1-a^n)-\sum_{j,k}\log \vert a-\overline{\omega_j}\omega_k\vert.\]

Using $\frac{d}{da}\log|a-z|=\Re\frac{1}{a-z}$ and the fact that $\sum_{j,k}\!\left(\frac{1}{a-\overline{\omega_j}\omega_k}\right)$ is real, we compute
\[
\begin{split}
a g_\omega'(a)
&=
-\frac{n^2a^{n}}{1-a^n}
-
\sum_{j,k}\!\left(\frac{a}{a-\overline{\omega_j}\omega_k}\right)\\
&=
-\frac{n^2}{1-a^n}
+
\sum_{j,k}\!\left(1-\frac{a}{a-\overline{\omega_j}\omega_k}\right)\\
&=
-\frac{n^2}{1-a^n}
+
\sum_{j,k}\!\left(\frac{1}{1-a\overline{\omega_j}\omega_k}\right)\\
\end{split}
\]
It follows from Corollary \ref{cor:key} that $g_\omega'(a)\ge 0$ for all $a\in[0,1)$.
Since $g_\omega(0)=0$, we conclude that $g_\omega(\rho^{-2})\ge 0$, which proves
\eqref{ineqmain3} and hence the desired inequality \eqref{ineqmain2}.

Now assuming that $z_1,\dots,z_n$ realize equality in \eqref{ineqmain2},
they also realize equality in \eqref{ineqmain4}. Hence $z_j=\rho \omega_j$
with $\omega_j\in \Bbb T$ ($j=1,\dots,n$) and $g_\omega(\rho^{-2})=0$.
Since $g_\omega$ is nondecreasing and $g_\omega(0)=g_\omega(\rho^{-2})=0$,
it follows that $g_\omega$ is constant on $[0,\rho^{-2}]$, hence
$g_\omega'(\rho^{-2})=0$, that is, 
\[
\sum_{j,k}\frac{1}{1-\rho^{-2}\overline{\omega_j}\omega_k}
=\frac{n^2}{1-\rho^{-2n}}.
\]
By Corollary \ref{cor:key}, the $\omega_j$ are the vertices of a regular $n$-gon inscribed in $\Bbb T$.

Conversely, if the $z_j$ are the vertices of a regular $n$-gon inscribed in $\{z\in \Bbb C :|z|= \rho\}$, then
\[
P(z)=\prod_{j=1}^n(z-z_j)=z^n+\Lambda, \qquad |\Lambda|=\rho^n.
\]
Hence
\[
\prod_{j,k}\bigl|1-\overline{z_j}z_k\bigr|
=\rho^n\prod_{j=1}^n\bigl|\overline{z_j}^{-n}+\Lambda\bigr|
=\rho^n\prod_{k=1}^n\bigl|-\overline{\Lambda}^{-1}+\Lambda\bigr|
=(\rho^{2n}-1)^n.
\]

% =========================
% Proof of auxiliary theorem
% =========================
\section{Proof of Theorem~\ref{thm:additive}}

%\begin{proof}
Recall that the Cauchy integral formula gives, for a holomorphic function $F$
in a neighborhood of $\overline{\mathbb{D}}$ and for $\zeta\in\mathbb{D}$,
\begin{equation}\label{Cauchy}
\frac1{2\pi}\int_0^{2\pi}\frac{e^{it}F(e^{it})}{e^{it}-\zeta}\,dt
=\frac{1}{2\pi i}\int_{|z|=1}\frac{F(z)}{z-\zeta}\,dz
=F(\zeta).
\end{equation}
Fix $z_1,\dots,z_n$ in $\mathbb{D}$ and define the folowing functions, holomorphic in a neighborhood of $\overline{\mathbb{D}}$:
\[f(z)=\sum_{j=1}^n \frac{1}{1-\overline{z_j}z},\ \ B(z)=\prod_{j=1}^n \frac{z-z_j}{1-\overline{z_j}z}.\]
Note that for all $t\in [0,2\pi]$,
\[|f(e^{it})|^2=\sum_{j,k}
\frac{e^{it}}{(1-\overline{z_j}e^{it})(e^{it}-z_k)}.
\]
%\begin{equation*}
%f(z)=\sum_{j=1}^n \frac{1}{1-\overline{z_j}z},
%\end{equation*}
%and the finite Blaschke product
%\begin{equation*}
%B(z):=\prod_{j=1}^n \frac{z-z_j}{1-\overline{z_j}z}.
%\end{equation*}
%Both are holomorphic in a neighborhood of $\overline{\mathbb{D}}$.
By \eqref{Cauchy}, we have
\begin{equation}\label{f2}
\frac1{2\pi}\int_0^{2\pi} \vert f(e^{it})\vert^2\,dt=\sum_{j,k}\frac1{2\pi}\int_0^{2\pi}\frac{e^{it}}{(1-\overline{z_j}e^{it})(e^{it}-z_k)}\,dt=\sum_{j,k}\frac1{1-\overline{z_j}z_k},
\end{equation}
\begin{equation}\label{Boverlinef}
\frac1{2\pi}\int_0^{2\pi} B(e^{it})\,\overline{f}(e^{it})\,dt
=\sum_{j=1}^n \frac1{2\pi}\int_0^{2\pi} B(e^{it})\,\frac{e^{it}}{e^{it}-z_j}\,dt
=\sum_{j=1}^n B(z_j)=0,
\end{equation}
\begin{equation}\label{B(0)}
\frac1{2\pi}\int_0^{2\pi}B(e^{it})\,dt=B(0)=\Lambda,
\end{equation}
and
\begin{equation}\label{f(0)}
\frac1{2\pi}\int_0^{2\pi} f(e^{it})\,dt=f(0)=n.
\end{equation}
We multiply the complex conjugate of \eqref{Boverlinef} by $\Lambda$ and we subtract it from (\ref{f(0)}) to find 
\begin{equation}\label{f-ZoverlineB}
\begin{split}
n&=\frac1{2\pi}\int_0^{2\pi} f(e^{it})\left(1-\Lambda\overline{B}(e^{it})\right)\,dt.
\end{split}
\end{equation}

%\frac1{2\pi}\int_0^{2\pi}\frac{e^{it}dt}{(1-\overline{z_j}e^{it})(e^{it}-z_k)}
%=\frac{1}{2\pi i}\oint_{|z|=1}\frac{dz}{(1-\overline{z_j}z)(z-z_k)}
%=\frac{1}{1-\overline{z_j}z_k}.
%\]
%Summing over $j,k$ gives 
%\begin{equation}\label{normf}
%\sum_{j,k}\frac1{1-\overline{z_j}z_k}=\frac1{2\pi}\int_0^{2\pi} \vert f(e^{it})\vert^2dt.
%\end{equation}
Besides, for all $t\in [0,2\pi]$, since $\vert B(e^{it})\vert=1$, we have 
\[
\left\vert 1-\Lambda \overline{B}(e^{it})\right\vert^2=1+\vert \Lambda\vert^2-\overline{\Lambda}B(e^{it})- \Lambda\overline{B}(e^{it}).
\]
Integrating and using \eqref{B(0)} yields 
\begin{equation}\label{normB}
\begin{split}
\frac1{2\pi}\int_0^{2\pi} \left\vert 1-\Lambda\overline{B}(e^{it})\right\vert^2dt=1-\vert \Lambda\vert^2.
\end{split}
\end{equation}
%Cauchy's formula gives first
%\begin{equation}\label{cauchy1}
%n=f(0)=\frac1{2\pi}\int_0^{2\pi} f(e^{it})\,dt,
%\end{equation}
%next, 
%On the unit circle, $\overline{f}(e^{it})=\sum_{j=1}^n \frac{e^{it}}{e^{it}-z_j}$.
Now \eqref{f2}, \eqref{f-ZoverlineB}, \eqref{normB} and Cauchy--Schwarz applied to the integral in \eqref{f-ZoverlineB} gives the desired inequality:
\begin{equation}\label{CS}
\begin{split}
n^2&\le
%\Big(\frac1{2\pi}\int_0^{2\pi}|1-Z\overline{B(e^{it})}|^2dt\Big)\Big(\frac1{2\pi}\int_0^{2\pi}|f(e^{it})|^2dt\Big)
(1-\vert \Lambda \vert^2)\sum_{j,k}\frac1{1-\overline{z_j}z_k}.
\end{split}
\end{equation}
We now turn to the equality case. Equality in \eqref{CS} holds if and only if
\begin{equation}\label{equ}
f=c (1-\overline{\Lambda}B),
\end{equation}
for some constant $c$. Evaluating in $0$, we find that
\[c=\frac n{1-\vert \Lambda\vert^2}.\]
Let $e_m=e_m(z_1,\dots,z_n)$ be the elementary symmetric polynomials:
\[
e_m=\sum_{1\le j_1<\cdots<j_m\le n}z_{j_1}\cdots z_{j_m},
\]
Put
\[
\begin{split}
P(z)&=\prod_{j=1}^n(z-z_j)=\sum_{m=0}^n (-1)^{\,n-m}e_{n-m}\,z^m\\
Q(z)&=\prod_{j=1}^n(1-\overline{z_j}z)=\sum_{m=0}^n (-1)^m\overline{e_m}\,z^m,\\
R(z)&=\sum_{j=1}^n\prod_{k\ne j}(1-\overline{z_k}z)=\sum_{m=0}^{n-1}(-1)^m (n-m)\,\overline{e_m}\,z^m,
\end{split}
\]
where the last identity follows by counting: each monomial of degree $m$ in the $\overline{z_j}$
is omitted by exactly $m$ indices and therefore appears exactly $n-m$ times.

We have 
\[
B=\frac{P}{Q},\ f=\frac{R}{Q}.
\]
Multiplying (\ref{equ}) by $(1-\vert \Lambda\vert^2)Q$, we obtain 
\[
\begin{split}
(1-\vert \Lambda\vert^2)\sum_{m=0}^{n-1}(-1)^m (n-m)\,\overline{e_m}\,z^m&=n\big(\sum_{m=0}^n (-1)^m\overline{e_m}\,z^m-\overline{\Lambda}\,\sum_{m=0}^n (-1)^{\,n-m}e_{n-m}\,z^m\big)\\
&=n\sum_{m=0}^n \big((-1)^m\overline{e_m}\,-\overline{\Lambda} (-1)^{n-m}e_{n-m}\big)z^m\\
\end{split}
\]
Identifying the coefficients in front of $z^m$ for $m=1,\dots,n-1$, we get 
\begin{equation}\label{1}
(n-(1-\vert \Lambda\vert^2)(n-m))(-1)^m\overline{e_m}=n \overline{\Lambda}(-1)^{n-m} e_{n-m}.
\end{equation}
Also replacing $m$ by $n-m$ and taking the conjugate,
\begin{equation}\label{2}
(n-(1-\vert \Lambda\vert^2)m)(-1)^{n-m}e_{n-m}=n \Lambda(-1)^m\overline{e_{m}}.
\end{equation}
Identities (\ref{1}) and (\ref{2}) yield to 
\[
\begin{split}
(n-(1-\vert \Lambda\vert^2)m)(n-(1-\vert \Lambda\vert^2)(n-m))\overline{e_m}&=n^2\vert \Lambda\vert^2\overline{e_m}\\
\Leftrightarrow (1-\vert \Lambda\vert^2)^2m(n-m)\overline{e_m}&=0.
\end{split}
\]
We deduce that $e_m=0$ for $m=1,\dots,n-1$ and consequently that 
\[P(z)=z^n+(-1)^ne_n=z^n+\Lambda.\]
Therefore $z_1,\dots,z_n$ are precisely the roots of $z^n+\Lambda$.

Conversely, if $z_1,\dots,z_n$ are the roots of $z^n+\Lambda$, then $e_m=0$ for $1\le m\le n-1$ and
the above coefficient comparison can be reversed to obtain \eqref{equ}, hence equality in \eqref{CS}.
%\end{proof}

\end{document}